\documentclass[a4paper,leqno,12pt]{amsart}

\usepackage{epsf,graphicx,epsfig}
\usepackage{amscd}
\usepackage{amsmath,latexsym,amssymb,amsthm}
\usepackage{setspace,cite}
\usepackage{lscape,fancyhdr,fancybox}
\usepackage{graphicx}

\usepackage{color}
\usepackage{url}
\usepackage{enumerate}
\usepackage[mathscr]{euscript}

  \theoremstyle{plain}

\swapnumbers
    \newtheorem{thm}{Theorem}[section]
    \newtheorem{prop}[thm]{Proposition}
   \newtheorem{lemma}[thm]{Lemma}
    \newtheorem{corollary}[thm]{Corollary}
    
    \newtheorem{subsec}[thm]{}
\theoremstyle{definition}
    \newtheorem{defn}[thm]{Definition}
\theoremstyle{remark}

\newenvironment{myeq}[1][]
{\stepcounter{thm}\begin{equation}\tag{\thethm}{#1}}
{\end{equation}}

\renewcommand{\Im}{\operatorname{Im}}
\renewcommand{\ker}{\operatorname{Ker}}

\newcommand{\Hom}{\operatorname{Hom}}
\newcommand{\Aut}{\operatorname{Aut}}

\newcommand{\Id}{\operatorname{Id}}
\newcommand{\pr}{\operatorname{pr}}
\newcommand{\rel}{\operatorname{rel}}

\title{}
\author{}
\date{}
\usepackage{amssymb}

\begin{document}
\title{Finite group actions on Kan complexes}
 \author{Goutam Mukherjee}
\email{goutam@isical.ac.in}
\address{Department of Mathematics,
 Indian Statistical Institute, Kolkata 700108,
West Bengal, India.}

 \author{Swagata Sarkar}
\email{swagatasar@gmail.com}
\address{Department of Mathematics,
 Indian Statistical Institute, Kolkata 700108,
West Bengal, India.}

 \author{Debasis Sen}
\email{sen\_deba@math.haifa.ac.il}
\address{Department of Mathematics,
 University of Haifa,
 31905 Haifa, Israel.}

\date{\today}

\thispagestyle{empty}

\noindent
\begin{abstract} We study simplicial action of groups on one vertex Kan complexes. We show that  every semi-direct product of the fundamental group of
an one vertex Kan complex with a finite group can be simplicially realized. We also calculate the cohomology of the fixed point set of a finite $p-$group action 
on an one vertex aspherical Kan complex.
\end{abstract}
\maketitle
\noindent
\section{Introduction}

Aspherical Kan complexes are Kan complexes all whose higher homotopy groups are trivial. This concept is a generalization of the concept of $K(\pi,1)$ spaces in the simplicial context. It may be mentioned that $K(\pi,1)$ spaces play an important role in topology, geometry and group theory.

In Section \ref{s1} we give the basic definitions and properties of one vertex Kan complexes. 
The techniques we have used for studying the group actions here are mostly generalizations of techniques developed by Conner and Raymond to study group actions on aspherical manifolds in the papers \cite{CR1, CR2, CR3,CR4,CR5}.
One of the techniques involves lifting the group action to the universal cover. This is the main reason we introduce the notion of one vertex Kan complexes. Note that the geometric realization of an one vertex Kan complex is a connected CW complex.

In Section \ref{s2}, we start with the definition of a regular covering space of a one vertex Kan complex. There is not enough exposition of regular covering of Kan complexes available in literature, in contrast to universal covering. 
Therefore, keeping the general situation in mind, wherever possible, we have proved the results on regular covers, instead of the universal cover. 

In the next section, Section \ref{s3}, we study the lifting problem of a given action to a regular cover. Note that all actions we consider are simplicial.

Now, let $(K, \phi)$ be a one vertex Kan complex with only one $0$-simplex $\phi$ and fundamental group $\pi:=\pi_{1}(K, \phi)$. Let $G$ be a discrete group acting effectively on
$(K, \phi)$. Note that, in case of a simplicial action on an one vertex Kan complex, the fixed point set is always non-empty. 
Hence, we have a homomorphism $\varphi \colon G \rightarrow \Aut(\pi)$  given by $g \cdot [x] := [g \cdot x]$, where $g \in G$ and $[x] \in \pi$. This homomorphism is called the abstract kernel. 
In Theorem \ref{thm42}, we show that such an abstract kernel, $\varphi \colon G \rightarrow \Aut(\pi)$, arising out of a discrete group action on an one vertex Kan complex,
gives rise to a group extension
\begin{displaymath}
 1 \longrightarrow \pi  \longrightarrow L \longrightarrow G \longrightarrow 1.
\end{displaymath}

\noindent
This group extension, $L$, can be simplicially realized as a group of morphisms of the universal cover, $\widetilde{K}$, of $K$, in such a way that the action $(L, \widetilde{K})$
covers the action $(G, K)$. Moreover, there is a natural isomorphism between the isotropy groups $G_{p(x)}$ and $L_{x}$, where $x \in \widetilde{K}$ and $p$ is the covering map.

Conversely, we prove that given an extension, $L$, of the fundamental group $\pi$, of a one vertex Kan complex, by a {\em finite} group $G$:
\begin{displaymath}
 1 \longrightarrow \pi  \longrightarrow L \longrightarrow G \longrightarrow 1
\end{displaymath}
such that the extension $L$ is a semi-direct product of $\pi$ and $G$, there exists an one vertex Kan complex $Y$, such that $L$ acts on $Y$ - simplicially
realizing the extension (cf. Theorem \ref{thm43}).

\noindent
In the last section, we calculate the cohomology of the fixed point set of a finite $p-$group action on a one vertex, aspherical Kan complex, and, as a by-product, get a variant of Borel's Theorem \cite[Theorem 3.1]{CR2} 
for minimal, aspherical, one vertex Kan complexes.

\noindent
\section{Preliminaries}\label{s1}

In this section we set up notation and review various definitions and properties of simplicial objects (\cite{May}, \cite{GJ}). We begin with the definition of a simplicial set.

\begin{defn}
A simplicial set $X$ is a sequence $\{X_n\}_{n\geq 0}$ of sets, together with set-maps $ \partial_i \colon  X_n\rightarrow X_{n-1}$ and $s_i\colon X_n\rightarrow X_{n+1},$ $0\leq i \leq n,$
satisfying the following simplicial identities,
$$\partial_i \partial_j = \partial_{j-1} \partial_i,~ ~  \partial_i s_j = s_{j-1} \partial_i,~ \mbox{if}~~i<j,$$
$$~~\partial_j s_j =id = \partial_{j+1}s_j,$$
$$\partial_i s_j = s_j \partial_{i-1}, ~~i >j+1;~~ s_is_j = s_{j+1} s_i,~~i\leq j.$$
A simplicial map $f\colon X\rightarrow Y$ between two simplicial sets is a collection of maps $f_n \colon  X_n\rightarrow Y_n,$ $n\geq 0,$  commuting with $\partial_i$ and $s_i$.
\end{defn}

For a simplicial set $X$, elements of $X_n$ are called $n$-simplices. A simplex $x\in X_n$ is called degenerate if $x=s_ix^{\prime}$ for some $x^{\prime}\in X_{n-1}$, $0\leq i\leq n-1$. Otherwise $x\in X_n$ is called non-degenerate.
A simplicial set is said to be finite if all simplices above a fixed dimension are degenerate.

The simplicial set $\Delta[n]$, $n\geq 0$, is defined as follows. The set of $q$-simplices is
$$\Delta [n]_q = \{(a_0,a_1,\cdots,a_q) ;~~\mbox{where}~ a_i\in \mathbb{Z},~ 0\leq a_0\leq a_1\leq\cdots\leq a_q\leq n\}.$$
The face and degeneracy maps are defined by $$\partial_i(a_0,\cdots,a_q)=(a_0,\cdots,a_{i-1},a_{i+1},\cdots,a_q),$$ $$s_i(a_0,\cdots,a_q)=(a_0,\cdots,a_i,a_i,a_{i+1},\cdots,a_q).$$

We have simplicial maps $$\delta_i\colon \Delta[n-1]\rightarrow \Delta[n],~~\sigma_i\colon \Delta[n+1]\rightarrow \Delta[n],~~0\leq i \leq n,$$ defined by $\delta_i(\Delta_{n-1})=\partial_i(\Delta_n)$ and $\sigma_i(\Delta_{n+1})=s_i(\Delta_n)$.
\begin{defn}
 The boundary subcomplex $\partial \Delta[n]$ of $\Delta[n]$ is defined as the smallest subcomplex of $\Delta[n]$ containing the faces $\partial_i\Delta_n,~~ i=0,1,\cdots,n $. The $k$-th horn $\Lambda_{k}^{n}$ of $\Delta[n]$ is the subcomplex of $\Delta[n]$ which is generated by all the faces $\partial_i\Delta_n$ except the $k$-th face $\partial_k\Delta_n$.
\end{defn}

\begin{defn}
A simplicial set $X$ is called a Kan complex if for every collection of $(n+1)$-tuple of $n$-simplices $(x_0,\cdots,x_{k-1},\hat{x}_k,x_{k+1},\cdots,x_{n+1})$  satisfying the compatibility conditions $\partial_i x_j=\partial_{j-1}x_i$, $i<j,~i\neq k,~j\neq k$, there exists an $(n+1)$-simplex $x$ such that $\partial_i x=x_i,~i\neq k$.
\end{defn}
The defining condition of a Kan complex is equivalent to the following statement. Any simplicial map from the $k$-th horn $\Lambda_k^{n+1}$ to $X$ can be extended to $\Delta[n+1]$, where $n\geq 0$ and $0\leq k\leq n$.

\begin{defn}
 A Kan complex $X$ is said to be minimal if $\partial_i x=\partial_i y,~i\neq k$, implies $\partial_k x=\partial_k y$.
\end{defn}

Next we briefly recall the definitions of homology and cohomology of a simplicial set. For a simplicial set $X$, let $C_n(X)$ denote the quotient of the free abelian group generated by the $n$-simplices of $X$ by the subgroup generated by the degenerate $n$-simplices.
Define $d\colon C_n(X)\rightarrow C_{n-1}(X)$ by $d=\sum_{i=0}^n (-1)^i \partial_i.$ Then $\{C_*(X),d\}$ becomes a chain complex, called the normalized chain complex of $X$. Given an abelian group $A$, the normalized cochain complex $\{C^*(X;A),\delta\}$ is defined
by $C^n(X;A)=\Hom_{\mathcal{A}b}(C_n(X),A)$ with differential $\delta\colon C^n(X;A)\rightarrow C^{n+1}(X;A)$, given by $\delta f=(-1)^{n+1} f\circ d,~f\in C^n(X;A)$, where $\mathcal{A}b$ denote the category of abelian groups and group homomorphisms.
Then the homology and cohomology groups of $X$ with coefficients in $A$ are defined by $H_n(X;A):= H_n(C_*(X)\otimes A,d\otimes id)$, and, $H^n(X;A):=H^n(C^*(X;A),\delta)$, respectively.

\begin{defn}
 The cartesian product $X\times Y$ of two simplicial sets $X$ and $Y$ is defined by $(X\times Y)_n=X_n\times Y_n$ with the face and degeneracy maps given by $$\partial_i(x,y)=(\partial_ix,\partial_iy)~ \mbox{and}~s_i(x,y)=(s_i x,s_i y).$$
\end{defn}

We define the concept of homotopy on simplicial sets:

\begin{defn}
 Let $f,g\colon X\rightarrow Y$ be simplicial maps. Then $f$ is said to be homotopic to $g$, written as $f\simeq g$, if there is a simplicial map $\mathcal{H}\colon X\times \Delta[1]\rightarrow Y$ such that $$\mathcal{H}\circ (id \times \delta_1)=f,~ \mathcal{H}\circ (id\times \delta_0)=g,$$ where we identity $X\times \Delta[0]$ with $X$ and $\delta_0,\delta_1\colon \Delta[0]\rightarrow \Delta[1]$ are the simplicial maps as defined earlier.

Suppose that $X^{\prime}$ and $Y^{\prime}$ are subcomplexes of $X$ and $Y$ respectively such that $f,g$ take $X^{\prime}$ into $Y^{\prime}$. If $f|_{X^{\prime}}=g|_{X^{\prime}}$ (=$\alpha$, say) then a homotopy $\mathcal{H}\colon f\simeq g$ is called a relative homotopy if $\mathcal{H}\circ (i\times id) =\alpha \circ \pr_1,$ where $\pr_1 \colon X^{\prime}\times \Delta[1] \rightarrow X^{\prime}$ is the projection onto the first factor and $i\colon X^{\prime}\hookrightarrow X$ is the inclusion. In this case we write $f\simeq g~(\rel~X^{\prime}).$ Intuitively, the homotopy leaves the restrictions of $f$ to $X^{\prime}$ unchanged.
\end{defn}

Note that the homotopy relation may in general fail to be an equivalence relation on the set $\Hom_{\mathcal{S}}(X,Y)$. But, homotopy is an equivalence relation on $\Hom_{\mathcal{S}}(X,Y)$ if $Y$ is a Kan complex \cite{May}. Thus we have the notions of homotopy equivalence, contractibility, etc., of simplicial sets. 

Let $X$ be a simplicial set and $\phi \in X_0.$ Then $\phi$ generates a subcomplex of $X$ which has exactly one simplex $s_{n-1}\cdots s_0(\phi)$ in dimension $n$. We will write $\phi$ unambiguously to denote either this subcomplex or any of its simplices.

\begin{defn}
 For a Kan complex $X$ and $\phi \in X_0$, define $$\pi_n(X,\phi):=\Hom_{\mathcal{S}}((\Delta[n],\partial\Delta[n]),(X,\phi))/\simeq (\rel ~\partial\Delta[n]),~n\geq 0.$$
\end{defn}

In general $\pi_0(X,\phi)$ is just a set. For $n\geq 1$, $\pi_n(X,\phi)$ is a group and it is abelian for $n>1$. One calls $\pi_1(X,\phi)$ the fundamental group of $X$.

\begin{defn}
 Given a group $\pi$ and a non-negative integer $n$, a Kan complex $X$ is called an Eilenberg-Mac~Lane complex of type $(\pi,n)$ if $\pi_n(X,\phi)=\pi$ and $\pi_i(X,\phi)=0$ for $i\neq n$. Such a complex is called a $K(\pi,n)$-complex if it is minimal.
\end{defn}

Observe that $\pi$ has to be abelian if $n> 1$. It is well known that any two $K(\pi,n)$ complexes are isomorphic and $K(\pi,n)_n=\pi$.

Note that the simplicial group $\pi_q=\pi,q\geq 0$, with all face and degeneracy maps the identity, is a $K(\pi,0)$-complex.

\begin{defn}
An one vertex Kan complex $(K,\phi)$ is said to be \emph{aspherical} if it is an Eilenberg-Mac~Lane complex of type $(\pi, 1)$, or equivalently, if it's universal cover is contractible.
\end{defn}

The Eilenberg-Mac~Lane complexes classify simplicial cohomology in the following sense.
\begin{thm}[\cite{May}]
 For a simplicial set $X$ and an abelian group $A$, there is natural bijection $$H^n(X;A)\leftrightarrow [X,K(A,n)].$$ Here $[X,K(\pi,n)]$ denotes the homotopy class of simplicial maps from $X$ to $K(A,n)$.
\end{thm}

\noindent
\section{Regular Cover of a Kan Complex}\label{s2}

In this section we introduce the notion of regular cover of a one vertex Kan complex, generalizing the universal cover \cite{gugg}. We begin with the definition of a covering map on a Kan complex.
\begin{defn}
A simplicial map between two complexes $p \colon E \rightarrow B$ is called a \emph{Kan fibration} if for every collection of $n+1$ compatible $n$-simplices $x_0 ,\cdots, \hat{x}_{k} ,\cdots, x_{n+1}$ of $E$, i.e., $\partial_{i} x_{j} = \partial_{j-1}
x_{i}$ for $i < j$ and $i,j \neq k$, which lie above the faces of some $(n+1)$-simplex $z$ of $B$, i.e., $\partial_{i} z = p(x_{i}) $ for $i \neq k$, there exists an $(n+1)$-simplex $y$ of $E$ such that $p(y) = z$ and $\partial_{i} y = x_{i} $ for $ i \neq k$.
If, in addition, we require the above simplex $y$ to be unique, then the map $p$ is called a \emph{covering map}.
\end{defn}

\noindent
In the above definition, $E$ is called the total complex, $B$ the base complex, and $F_{\phi } = p^{-1} (\phi_{B} )$ is called the fibre over $\phi_{B}$, where $\phi_{B}$ is the complex generated by a vertex $\phi \in B_{0}$. The triple $(E,p,B)$ is called a fibre space.

\noindent
Let $K$ be a one vertex Kan complex, with only one $0$-simplex $\phi$ and fundamental group $\pi = \pi_1 (K, \phi )$. Let $H \leq \pi$ be a normal subgroup of $\pi$. 
Then the Kan complex which is the covering complex of $K$ corresponding to $H$ is described as follows :

\begin{defn}
Define a complex $\widetilde{K}_H$ by ${(\widetilde{K}_H)}_{n} = K_{n} \times \pi / H$ with face and degeneracy operators defined by:
\begin{itemize}
\item $\partial_i (x, \overline{\alpha})$ = $(\partial_i x,  \overline{\alpha})$, $0\leq i < n,$
       $\partial_n(x,\overline{\alpha}) =( \partial_n x , [ {\partial_{0}}^{n-1}x]^{-1}\overline{\alpha} )$,
\item  $s_i (x,\overline{\alpha})$ = $(s_i x, \overline{\alpha})$, $0\leq i\leq n,$
\end{itemize}
where $x\in K_n$ and $ \overline{\alpha}$ denotes the coset in $\pi/H$ represented by $\alpha\in \pi$.
\end{defn}
\noindent
\begin{prop}
With notations as above, $\widetilde{K}_H$ is a Kan complex
with fundamental group $H$ and the projection $p\colon \tilde{K}_H\rightarrow K,~p(x,\overline{\alpha})=x,$ is a covering map.
\end{prop}
\begin{proof}
Since $K$ is a Kan complex, most of the simplicial identities for $\widetilde{K}_H$ follow trivially from the simplicial identities for $K$,
 and it suffices to verify that $\partial_{n+1} s_n = \Id_{(\widetilde{K}_H)_{n}}$.
We take $(x, \overline{\alpha} ) \in (\widetilde{K}_H)_{n}$ and since we know that $\partial_{0}^{n} x = \phi$ we have, $\partial_{n+1} s_n (x, \overline{\alpha})$ = $(\partial_{n+1} s_{n} x, [\partial_{0}^{n} s_n x]^{-1} * \overline{\alpha})$
= $(x, [s_{0} \partial_{0}^{n}x]^{-1}  * \overline{\alpha} )$ = $(x, \overline{\alpha} )$, and hence we are done.

\noindent
Next we prove that $\widetilde{K}_H$ is also a Kan complex. 

\noindent
Consider a $(n+1)$-tuple  $(\tilde{x}_{0} ,\cdots , \tilde{x}_{k-1} , \tilde{x}_{k+1} ,\cdots \tilde{x}_{n+1})$ of compatible $n$-simplices in $\widetilde{K}_H$, 
where $\tilde{x}_i = (x_i, \overline{\alpha}_i)\in (\widetilde{X}_H)_n$.
Then we have, $\partial_{i} \tilde{x}_{j} = \partial_{j-1} \tilde{x}_{i}$ for $i < j$ and $i,j \neq k$. 
Therefore the $(n+1)$-tuple of $n$-simplices in $K$ $(x_{0},\cdots, x_{k-1}, x_{k+1},\cdots, x_{n+1})$ is compatible and hence, there exists an extension $y \in K_{n+1}$.
To check the compatibility of the second entry we consider the following cases:
\noindent
For $k = n+1$, taking $i,j \leq n$ and $i = j-1$ and applying the condition $\partial_{i} \tilde{x}_{j} = \partial_{j} \tilde{x}_{i}$, we get $\overline{\alpha}_{j-1}  =\overline{\alpha} _{j}  = \overline{\alpha}_{0} $ (say) for all $j \leq n$. Hence, $\tilde{y} = (y, \overline{\alpha}_{0} )$ is the required extension.

\noindent
For $k=n$, we take $j=n+1$ and $i<j-1$. Then the compatibility conditions give $\overline{\alpha}_{n+1}=[\partial_n^{n-1}x_i]^{-1}\overline{\alpha}_i $. Further, applying compatibility conditions for $i<j<n$, we get $\overline{\alpha}_{0} = \overline{\alpha}_{1}  =\cdots= \overline{\alpha}_{n-1}$. Then $\tilde{y}=(y,\overline{\alpha}_0)$ is the required extension.
To see this, it suffices to check that $\partial_{n+1}\tilde{y}=(x_{n+1},\overline{\alpha}_{n+1})$. Now, $\partial_{n+1}\tilde{y}=(\partial_{n+1}y,[\partial_0^n y]^{-1}\overline{\alpha}_0)= (x_{n+1},[\partial_0^{n-1}x_0]^{-1}\overline{\alpha}_0 )=(x_{n+1},\overline{\alpha}_{n+1}).$
For the case $k < n$, we have $[\partial_0^{n+1}x_{n+1}]^{-1}\overline{\alpha}_{n+1}=[\partial_0^{n-1}x_n]^{-1}\overline{\alpha}_n $, using the compatibility conditions for $i=n,j=n+1$. Moreover, if $i<j\leq n$, then we get $\overline{\alpha}_0 =\cdots=\overline{\alpha}_{k-1}=\overline{\alpha}_{k+1}=\cdots =\overline{\alpha}_n$.
Therefore, we take $\tilde{y}=(y,\overline{\alpha}_n)$. To prove that it is the desired extension, it suffices to check that $\partial_{n+1}\tilde{y}=(x_{n+1},\overline{\alpha}_{n+1})$, i.e $[\partial_0^n y]^{-1}\overline{\alpha}_n =\overline{\alpha}_{n+1}$. But we know that
$\overline{\alpha}_{n+1}=[\partial_0^{n+1}x_{n+1}][\partial_0^{n-1}x_n]^{-1}\overline{\alpha}_n = [\partial_0^{n-1}\partial_{n+1}y][\partial_0^{n-1}\partial_n y]^{-1}\overline{\alpha}_n $. Applying simplicial identities, we have $\partial_0^{n-1}\partial_{n+1}y=\partial_2\partial_0^{n-1}y$ and
$\partial_0^{n-1}\partial_n y=\partial_1\partial_0^{n-1}y.$ Therefore, $[\partial_1 z]=[\partial_0 z][\partial_2 z]$ in $\pi$, where $z=\partial_0^{n-1}y$. Hence, $\overline{\alpha}_{n+1}=[\partial_2z][\partial_1 z]^{-1}\overline{\alpha}_n =[\partial_0 z]^{-1}\overline{\alpha}_n =[\partial_0^ny]^{-1}\overline{\alpha}_n $.

With similar computations as above shows that the projection map $p$ is a covering map.

Finally, we calculate the fundamental group of $\widetilde{K}_H$. Consider the homotopy exact sequence of the Kan fibration $p\colon \widetilde{K}_H \rightarrow K $.
Let $F$ be the fibre over $\phi$ and let $\tilde{\phi} = (\phi , \overline{e})$ be the point in the fibre which corresponds to the identity $\overline{e}$ of $\pi / H$. This fibre $F$ satisfies $\pi_{n} (F, \tilde{\phi}) = 0$ for $n > 0$.
Now, if there exists a $1$-simplex from the $0$-simplex $(\phi , \overline{\alpha})\in F_0$ to the $0$-simplex $(\phi ,\overline{\beta} )\in F_0$ of the fibre $F$ over $\phi$, then there exists an element $ \overline{\gamma} \in \pi /H$ such that $\partial_{1} (s_{0} \phi , \overline{\gamma}) = (\phi ,\overline{\alpha} )$
and $\partial_{0} (s_{0} \phi , \overline{\gamma}) = (\phi , \overline{\beta})$. But, $\partial_{0} (s_{0} \phi , \overline{\gamma}) = (\phi , \overline{\gamma}) = \partial_{1} (s_{0} \phi , \overline{\gamma})$. Therefore, $\overline{\alpha} = \overline{\beta}$ and this implies that $\pi_{0} (F, \tilde{\phi}) = \pi /H$.
Consider the boundary map $\partial_{H} : \pi_{1} (K,\phi) \rightarrow \pi_{0} (F, \tilde{\phi})$ in the homotopy long exact sequence of the Kan fibration $p\colon\widetilde{X}_H\rightarrow X$.
If $\alpha = [x] \in \pi_{1} (K,\phi)$, where $x \in K_{1}$, then, if we choose $(x,\overline{\beta} )$ such that $\partial_{1} (x, \overline{\beta}) = \tilde{\phi}  $, we have $\partial_{1} (x, \overline{\beta}) = (\partial_{1} x, [x]^{-1} \overline{\beta}) = (\phi, \overline{\alpha}^{-1}\overline{\beta} ) =  (\phi, \overline{e})$.
Hence, $\partial_{H} (\alpha ) = [\partial_{0} (x,\overline{\beta} )] = (\partial_{0} x, \overline{\beta}) = (\phi ,\overline{\beta} ) = (\phi, \overline{\alpha})$ and therefore $\partial_H\colon \pi=\pi_1(X,\phi)\rightarrow \pi_0(F,\tilde{\phi})=\pi/H$ is just the projection map.
Using the homotopy exact sequence of the Kan fibration $p:\widetilde{K}_H \rightarrow K $, we get $\pi_{1} (\widetilde{K}_H , \tilde{\phi}) = \ker \partial_{H} = H$.
\end{proof}

\begin{defn}
The Kan complex $\widetilde{K}_{H}$, together with the map $p\colon \widetilde{K}_{H}\rightarrow K$, is called the \emph{regular cover} of the one vertex Kan complex $K$, corresponding to the normal subgroup $H$ of the fundamental group of $K$.
\end{defn}
The $0$-simplex $(\phi, \overline{e} )$ of $\widetilde{K}_{H}$ is denoted by $\tilde{\phi}$.
The group $\pi /H$ acts simplicially on the right of the regular cover $\widetilde{K}_H$ by multiplication on the second factor,
 \begin{myeq}\label{eq}
\psi\colon \pi/H\rightarrow\Aut{\widetilde{K}_H},~~(k, \overline{\alpha})\overline{\beta}:=(k,\overline{\alpha} \overline{\beta}),
\end{myeq}
 for all $(k,\overline{\alpha}) \in \widetilde{K}_H$ and $\overline{\beta} \in \pi/H$.

Throughout the paper, we shall use the notations of this section.

\section{Lifting the action of a group to a regular cover}\label{s3}

In this section we discuss the problem of lifting a simplicial group action on a one vertex Kan complex to a regular cover.

Let $G$ be a discrete group acting simplicially on a one vertex Kan complex $(K,\phi)$, with vertex $\phi$.
The action of $G$ induces a homomorphism $G \longrightarrow \Aut( \pi )$, given by $g\mapsto g_*,$ with $g_*$  defined by $g_*[x] := [gx]$, where $g \in G$ and $[x] \in \pi$.
\begin{thm} \label{thm32}
Let $H \leq \pi_1 (K, \phi)$ be a normal subgroup invariant under the action of $G$ on $\pi_1 (K, \phi)$. Then the $G-$action on $K$ can be lifted to a $G-$action on $K_H$ such that the covering map $p: K_H \longrightarrow K$ is $G-$equivariant. 

Moreover, for all $g \in G$, $b\in \widetilde{K}_{H}$ and $ \overline{\alpha}\in \pi_1( K , \phi)/H$
\begin{displaymath}
 g(b \cdot \overline{\alpha})= (gb)g_{*}(\overline{\alpha}).
\end{displaymath}

\end{thm}
\begin{proof}
Since the action of $G$ preserves $H$, i.e., $g(H) = H$, for all $g \in G$,  we have a homomorphism $\varphi \colon G \rightarrow \Aut( \pi  /H) $. By abuse of notation, we also denote $\varphi(g)$ by $g_*$.
Using this we lift the action of $G$ on $K$ to an action of $G$ on the regular cover $\widetilde{K}_H$ of $K$.
If $g \in G$, $k \in K$ and $\overline{\alpha} \in \pi / H$, then the action of $G$ on $(\widetilde{K}_H)_n=K_n \times \pi /H$ is the diagonal action given by
\begin{myeq}\label{eq1}
g\cdot (k,\overline{\alpha})= (gk, g_*(\overline{\alpha})).
\end{myeq}
We check that this action is well-defined. Let $(k, \overline{\alpha}) \in {(\widetilde{K}_H)}_{n}$. It is enough to check that $\partial_{n}$ commutes with the action of $g$ on $\widetilde{K}_H$.
We have,
$$\begin{array}{lll} \partial_{n} (g(k , \overline{\alpha}))
&=& \partial_{n} (gk , g_*( \overline{\alpha})) \\
&=& (\partial_{n} (gk) , [g  \partial_{0}^{n-1} k ]^{-1} g_*(\overline{\alpha}) )\\
&=& (g \partial_{n} k , g_*([\partial_{0}^{n-1} k]^{-1}) g_*(\overline{\alpha})) \\
&=&  (g \partial_{n} k , g_*([\partial_{0}^{n-1} k]^{-1} \overline{\alpha} ))\\
&=& g \cdot ( \partial_{n} k , [\partial_{0}^{n-1} k]^{-1}  \overline{\alpha})  \\
&=& g \cdot \partial_{n} (k, \overline{\alpha}).
\end{array}$$

Then the right action of $\pi /H$ on $\widetilde{K}_H$ is related to the left action of $G$ as follows:
$g\cdot((k,\overline{\alpha})\overline{\beta}) = g.(k, \overline{\alpha \beta}) = (gk, g_{*} (\overline{\alpha \beta})) = (gk,  g_{*}(\overline{\alpha})) g_{*}(\overline{\beta})$
= $(g \cdot (k, \overline{\alpha})) g_{*} (\overline{\beta})$.
\end{proof}

Let $G$ be a group and and $\pi$ be a $G$-module. For each $g \in G$, let $g_{*}$ denote the corresponding
automorphism of $\pi$.
\begin{defn}
A \emph{crossed homomorphism} is a function $\varphi \colon G \rightarrow \pi$ satisfying the relation $\varphi(gh)$ = $\varphi(g)g_{*}(h)$, for all $g, h \in G$.
\end{defn}
To each crossed homomorphism we associate the graph $G_{\varphi} \subset \pi \rtimes G$, which is the image of the isomorphism $g \mapsto (\varphi(g), g)$.

We will denote by $G_{y}$ the isotropy subgroup at any $y \in K$. Let $b \in \widetilde{K}_H$. Since $\pi /H$ acts by deck transformations on $\widetilde{K}_H$, for any $g \in G_{p(b)}$,
there is a unique $ \overline{\alpha} \in \pi /H$ such that $gb = b\overline{\alpha}$. We define a crossed homomorphism $$r_{b} \colon G_{p(b)} \longrightarrow \pi /H,~~g\mapsto \overline{\alpha}$$
Let $g, g^{'} \in G_{p(b)}$. Then, $b r_{b}(gg^{'}) = (gg^{'})b = (g(g^{'}.b) = g(b r_{b}(g^{'}))) = (gb)(g_{*}(r_{b}(g^{'}))) = b r_{b}(g) g_{*} (r_{b}(g^{'}))$. Therefore, $r_{b} (gg^{'}) = r_{b}(g) g_{*} (r_{b}(g^{'})) $
and $r_{b}$ is indeed a crossed homomorphism with kernel,
\begin{displaymath}
\ker r_{b} = \{ g \in G_{p(b)} | r_{b}(g) = \overline{e}\in \pi/H \} = G_{b}.
\end{displaymath}
We have the following lemma.

\noindent
\begin{lemma} \label{lem33}
 For $b \in \widetilde{K}_H$, $\overline{\alpha} \in \pi /H$, $\hat{g} \in G$ we have:
\begin{enumerate}[(i)]
\item $r_{b \overline{\alpha}} = \overline{\alpha}^{-1} r_{b}(g)g_{*}(\overline{\alpha})$.
\item $r_{\hat{g} b}(\hat{g}g\hat{g}^{-1}) = \hat{g}_{*} r_{b}(g)$, for all $g \in G_{p(b)}$.
\end{enumerate}
\end{lemma}

\noindent
\begin{proof}
(i) Note that since $p(b) = p(b \overline{\alpha})$, $r_{b}$ and $r_{b \overline{\alpha}}$ are both defined on $G_{p(b)}$. We have
$b \overline{\alpha} r_{b \overline{\alpha}}(g) = g(b \overline{\alpha}) = (gb) g_{*}(\overline{\alpha}) = br_{b}(g)g_{*}(\overline{\alpha}) = b=\overline{\alpha}$ $\overline{\alpha}^{-1} r_{b}(g)g_{*}(\overline{\alpha})$.

\noindent
(ii) Again, note that $G_{p(\hat{g}b)} = G_{\hat{g}p(b)} = \hat{g} G_{p(b)} \hat{g}^{-1}$. Therefore all the maps are defined. Now, if $g \in G_{p(b)}$, then $gb = br_{b}(g)$.
Hence, $\hat{g}g\hat{g}^{-1}\hat{g}b = \hat{g}(gb) = \hat{g}(br_{b}(g)) = (\hat{g}b)(\hat{g}_{*}r_{b}(g))$.
\end{proof}

\noindent
Let $K^G \subset K$ denote the fixed point set of the action of $G$ on $K$ and let $E \subset \widetilde{K}_{H}$ be the fixed point set of the lifted action of $G$ on $\widetilde{K}_{H}$.
Then we have the following result:

\noindent
\begin{prop} \label{prop34}
 If $ \overline{\alpha}\in \pi /H$ is such that $E \cap E\overline{\alpha}$ is non-empty, then $g_{*}(\overline{\alpha}) =\overline{\alpha} $ for all $g \in G$ and $E = E\overline{\alpha}$.
\end{prop}

\noindent
\begin{proof} Note that for any $b \in E$,~$G_{b} = G$ and the crossed homomorphism $r_{b} \colon G \rightarrow \pi /H$ is trivial.
Since $E \cap E\overline{\alpha}$ is non-empty, choose $b \in E$ such that $b \overline{\alpha}\in E$. Then $r_{b\overline{\alpha}} \colon G \rightarrow \pi /H$ is also trivial. But, by previous lemma, we have
$r_{b\overline{\alpha}}(g) = \overline{\alpha}^{-1} r_{b}(g)g_{*}(\overline{\alpha})$. Therefore, $\overline{\alpha}^{-1}g_{*}(\overline{\alpha}) = 1$ for all $g \in G$. That is, $g_{*}(\overline{\alpha}) = \overline{\alpha}$ for all $g \in G$,
which also implies $E\overline{\alpha} = E$.
\end{proof}

\noindent
Let $\Gamma \subset \pi/H$  be the fixed point set of the action of $G$ on $\pi/H$, i.e.,
\begin{displaymath}
 \Gamma := \{ \overline{\alpha} \in \pi/H | g_{*}(\overline{\alpha}) = \overline{\alpha} ,  \forall g \in G \}.
\end{displaymath}
Then, since $E$ is invariant under the action of $\Gamma$ on $\widetilde{K}_H$, the previous result shows that
\begin{displaymath}
 E / \Gamma = \Im(\pi |_{E} : E \rightarrow K) = K^G.
\end{displaymath}

We end this section by combining the two actions on the regular cover of a one vertex Kan complex.
Let $\varphi \colon G \rightarrow \Aut(\pi/H)$ denote the homomorphism corresponding to the action of $G$ and for $g \in G$, let $g_{*}$ denote the automorphism $\varphi(g)$.
Then $\varphi$ can be used to form the semi direct product $G \rtimes_{\varphi} \pi/H$, whose undelying set is $G\times \pi/H$, with multiplication in the semi direct product given by
\begin{displaymath}
 (g, \overline{\alpha})\cdot (h, \overline{\beta}) := (gh, h_{*}^{-1}(\overline{\alpha})\overline{\beta}).
\end{displaymath}
\noindent
Then we combine the left covering action of $G$ on $\widetilde{K}_{H}$ with the right action of $\pi/H$, by deck transformations, to a right action of $G \rtimes_{\varphi} \pi/H$ as follows:
if $b = (k,\overline{\beta}) \in \widetilde{K}_{H}$ and $(g,\overline{\alpha}) \in G \rtimes \pi/H$, then
\begin{displaymath}
b \cdot (g, \overline{\alpha} )= (k,\overline{\beta})\cdot( g,\overline{\alpha}) =  (g^{-1}(k,\overline{\beta}))\overline{\alpha} = (g^{-1}k,g_{*}^{-1}(\overline{\beta})\overline{\alpha}).
\end{displaymath}

\noindent
Note that, if $b\cdot(g,\overline{\alpha}) = b$, then $(g^{-1}b)\overline{\alpha} =b$,that is, $g^{-1}b = b\overline{\alpha}^{-1}$, that is,  $ \overline{\alpha} = r_{b}(g)$. Therefore, at each point $b \in \widetilde{K}_{H}$, the isotropy
subgroup of the action of $G \rtimes \pi/H$ is the graph of the function $r_{b} \colon G_{p(b)} \rightarrow \pi/H$. We summarize the above in a lemma:

\begin{lemma} \label{lem35}
Let a group $G$ act simplicially on $K$ such that the normal subgroup $H \leq \pi$ is invariant under the action of $G$. Then, at each point $b \in \widetilde{K}_{H}$, there is a canonical crossed homomorphism
$r_{b} \colon G_{p(b)} \rightarrow \pi/H$, whose graph is the isotropy subgroup of the action of $G \rtimes_{\varphi} \pi/H$ on $\widetilde{K}_{H}$.
\end{lemma}

\section{Group extensions and regular covers}\label{s4}

In this section, we associate a split extension of $\pi/H$ by a discrete group $G$, acting on $K$, such that the extension group acts on $\widetilde{K}_H$, covering the given action of $G$. 
Conversely, if we start with a split extension of $\pi/H$ by a finite group $G$, then we can find a one vertex Kan complex with an action of $G$, such that if we apply the previous result then we obtain the given extension.

Let $N$ and $G$ be groups and let $\Aut(N)$ denote the group of automorphisms of $N$. Let $\varphi : G \rightarrow \Aut(N) $ denote a group homomorphism. Recall that we can define a group extension $L$ given by

\begin{displaymath}
 1 \longrightarrow N \longrightarrow L \longrightarrow G \longrightarrow 1,
\end{displaymath}
where $L = N \times G$ and multiplication in $L$ is defined as
$$
 (g,x)\cdot (h,y) = (g\varphi (x) (h), xy), ~~g,h \in N ~\mbox{and}~ x,y \in G.
$$
\noindent
Given the above situation, we have the following lemma.

\begin{lemma} \label{lem41}
Let $M$ be a group. If $h \colon N \rightarrow M$ is a group homomorphism, then $h$ can be extended to a homomorphism $H \colon L \rightarrow M$ if and only if
there exist a homomorphism $T \colon G \rightarrow M$ such that
$T(x)\cdot h(g) = h( \varphi (x) (g))\cdot T(x)$ , for all $x \in G$ and $g \in N$.
\end{lemma}

\noindent
\begin{proof} First suppose that there exists a homomorphism $T \colon G \rightarrow M$ such that $T(x)\cdot h(g) = h( \varphi (x) (g))\cdot T(x)$ , for all $x \in G$ and $g \in N$.
Then we can define $H \colon L \rightarrow M$ by $H(g,x)$ = $h(g) \cdot T(x)$. Since

$$\begin{array}{lll} H(g_{1} \cdot \varphi (x) g_{2} , x.y)
&=& h(g_{1})\cdot h (\varphi (x) (g_{2}) ) \cdot T(x) \cdot T(y) \\
&=&h(g_{1}) \cdot T(x) h(g_{2}) \cdot T(y)  \\
&=&H(g_{1} , x) \cdot H(g_{2} , x),
\end{array}$$
\noindent
for all $g_{1} , g_{2} \in N$ and $x,y \in G $, $H$ is the required homomorphism.

\noindent
Conversely, suppose an extended homomorphism $H \colon L \rightarrow M$ exists. We can define $T \colon G \rightarrow M$ by $T(x)$ := $H(0,x)$, where $0$ denotes the identity of $N$.
Then $T$ is obviously a homomorphism and for $x \in G$ and $g \in N$, we have
$$\begin{array}{lll} T(x) h(g)
&=& H ((0,x)\cdot (g,e_{G})) \\
&=& H ( \varphi (x) (g) , x)  \\
&=& H(\varphi (x)(g), e_{G} ) H(0,x) \\
&=& h(\varphi(x)(g))T(x)
\end{array}$$

\noindent
Hence the lemma.
\end{proof}

Let $G$ be a discrete group acting on $K$, such that the action by any element of $G$ preserves $H$, for some normal subgroup $H \leq \pi$ . Then we have a homomorphism $\varphi \colon G \rightarrow \Aut( \pi  /H) $.
This determines a split extension,
\begin{displaymath}
 1 \rightarrow \pi /H \rightarrow L \xrightarrow{\eta} G \longrightarrow 1.
\end{displaymath}
Recall that $\pi/H$ acts on $\widetilde{K}_H$ by the homomorphism $\psi$ (cf. Eq. \ref{eq}).
\begin{thm} \label{thm42}

 With notations as above, the $\pi/H$ action on $\widetilde{K}_H$ can be extended to a $L$-action on $\widetilde{K}_H$ such that under $\eta \times p \colon (L , \widetilde{K}_H) \rightarrow (G,K)$, where $p \colon \widetilde{K}_H \rightarrow K$ is the covering projection, $(L, \widetilde{K}_H)$ covers the action $(G,K)$.
 Furthermore, there is a natural isomorphism between the isotropy groups $G_{p(x)}$ and $L_{x}$, where $x \in \widetilde{K}_H$.
\end{thm}

\begin{proof}

Using Theorem \ref{thm32}, we lift the $G$-action to an action on the regular cover $\widetilde{K}_H$.
Recall from Eq. \ref{eq1} that this action $T \colon G \rightarrow \Aut(\widetilde{K}_H)$ is given by  $(x, \overline{\sigma}) \mapsto (f\cdot x , \varphi(f)\cdot \overline{\sigma})$. Then,

$$\begin{array}{lll} T(f)\psi(\overline{\alpha}) (x, \overline{\sigma})
&=& T(f) (x, \overline{\sigma} \overline{\alpha}) \\
&=& (fx , \varphi(f) (\overline{\sigma} \overline{\alpha})) \\
&=& (fx , \varphi(f)(\overline{\sigma})\varphi(f)(\overline{\alpha}))\\
&=& \psi(\varphi(f) (\overline{\alpha})) (fx , \varphi(f)\overline{\sigma} ) \\
&=& \psi(\varphi(f)(\overline{\alpha})) T(f)(x,\overline{\sigma}).
\end{array}$$

\noindent
Hence, by Lemma \ref{lem41} we have an extension $\Psi \colon L \rightarrow \Aut(\widetilde{K}_H) $,  which gives an action of $L$ on $\widetilde{K}_H$.

For $f \in G_{p(x)}$, choose unique $\overline{\alpha} \in \pi/H$ such that $T(f)(x) = \psi (\overline{\alpha}^{-1})(x)$. This gives a one-to-one correspondence between the isotropy groups $G_{p(x)}$ and  $L_{x}$. 
Also, since, for $(f, \overline{\alpha} ) $ and $(g, \overline{\beta}) \in L_{x} $,
 $$ \begin{array}{lll} x
&=& \psi(\overline{\alpha}  ) \cdot T(f) \cdot \psi (\overline{\beta}) \cdot T(g) (x)  \\
&=& \psi(\overline{\alpha} ) \cdot \psi (\varphi(f) (\overline{\beta})) \cdot T(f) \cdot T(g) (x) \\
&=& \Psi(\overline{\alpha} \cdot \varphi(f) (\overline{\beta})) \cdot T(f \cdot g)(x)
\end{array}$$
\noindent
the above correspondence between $G_{p(x)}$ and  $L_{x}$ is an isomorphism of groups.

\end{proof}

\noindent
Let us now consider the converse problem, viz., if we are given a group extension
\begin{displaymath}
 1 \longrightarrow \pi /H \longrightarrow L \longrightarrow G \longrightarrow 1
\end{displaymath}
of $\pi/H$ by a {\em finite} $G$, can we realize this extension simplicially by the previous construction? This can be done as follows provided that $L$ is the semi-direct product of $\pi /H$ and $G$, corresponding to some homomorphism $\varphi \colon G \rightarrow \Aut( \pi  /H) $.

\begin{thm} \label{thm43}
Given a split extension
\begin{displaymath}
 1 \longrightarrow \pi /H \longrightarrow L \longrightarrow G \longrightarrow 1,
\end{displaymath}
of $\pi /H $ by a finite group $G$, corresponding to a homomorphism $\varphi \colon G \rightarrow \Aut( \pi  /H) $, there is a one vertex Kan complex with $G$-action, such that  application of construction of Theorem \ref{thm42} on it yields the given extension.
\end{thm}

\begin{proof}
Since $G$ is finite, define $Y$ to be the set of all functions $\chi \colon G \rightarrow \widetilde{K}_H,$ written as the iterated Cartesian product with itself. Then $Y$ is also a Kan complex with one vertex.
Using the homomorphism $\varphi \colon G \rightarrow \Aut( \pi  /H) $, we can define an action of $\pi /H$ on $Y$ as follows:

To each $\overline{\sigma} \in \pi /H$ associate a morphism $\Sigma\colon Y\rightarrow Y$, defined by :

\begin{displaymath}
 \Sigma (\chi) (z) := \psi (\varphi (z) ( \overline{\sigma}) ) (\chi (z)), ~~\chi\in Y,~~z\in G.
\end{displaymath}

\noindent
Then the homomorphism $\theta\colon \pi /H \rightarrow \Hom(Y)$, which takes $ \overline{\sigma}\mapsto \Sigma$ defines a left action of $\pi /H$ on $Y$. This action is free since the action of $\pi/H$ on $\widetilde{K}_{H}$ is free.
We now define a homomorphism $J \colon G \rightarrow \Hom(Y)$ by associating to each $x \in G$ a simplicial morphism $J_x \colon Y \rightarrow Y$ given by $J_x (\chi) (z)$ := $\chi (xz)$. Since
$$(J_x \circ J_y ) (\chi ) (z)=J_x (\chi) (yz)=\chi (xyz)=J_{xy} (\chi ) (z),$$
\noindent
the map $J$ is a homomorphism. Further we have,

 $$ \begin{array}{lll} (J_x \circ \Sigma ) (\chi ) (z)
&=& J_x (\psi (\varphi (z)(\overline{\sigma}))(\chi (z))) \\
&=& \psi (\varphi (x) \varphi (z)(\overline{\sigma})) J_x (\chi) (z)\\
&=& \psi (\varphi (x) \varphi (z)(\overline{\sigma}))  (\chi) (xz).
\end{array}$$

\noindent
Therefore by Lemma \ref{lem41}, the action $\theta$ of $\pi /H$ on $Y$ extends to an action of $L$ on $Y$. Hence we have an action of $L/(\pi/H)=G$ on $Y/(\pi/H)$. This action of $G$ on $Y/(\pi/H)$ is such that applying the previous construction on this gives the action of $L$ on $Y$.
\end{proof}

\section{Action of a finite group on aspherical Kan complex}\label{s5}

In this section we will study aspherical Kan complexes and discuss the action of a finite group on an aspherical, finite, one vertex Kan complex.
Recall that a Kan complex is said to be finite if all simplices above a certain dimension are degenerate and an aspherical one vertex Kan complex is an Eilenberg-Maclane complex.

In the previous section, we have seen that if $G$ is a finite group acting on a one vertex Kan complex $(K,\phi)$,
with fundamental group $\pi$, then we have a homomorphism $\varphi \colon G \rightarrow \Aut(\pi) $. This homomorphism is called the
associated abstract kernel and gives rise to a group extension:

\begin{displaymath}
 1 \longrightarrow \pi  \longrightarrow L \longrightarrow G \longrightarrow 1.
\end{displaymath}

In other words, we have shown that every abstract kernel can be algebraically realized. We now consider the specific case where $G$ is a $p-$group and $K$ is a finite, aspherical one vertex Kan complex. 
We denote the fixed point set of $G$-action by $K^G$.

\begin{thm}\label{thm52}
 Let $G$ be a finite $p$-group, $p$ prime, acting on aspherical one vertex Kan complex $K$ such that $H^{*} (K;\mathbb{Z}_{p})$ is
finite dimensional as a vector space over $\mathbb{Z}_{p}$. Let $K^{G}$ denote the fixed point set of this action.
Then we have the following:
\begin{enumerate}[(i)]
\item $K^G$ itself is a one vertex Kan complex, with only one $0$-simplex $\phi$.
\item $H^{*}(K^G; \mathbb{Z}_{p}) \cong H^{*}(\Gamma ; \mathbb{Z}_{p})$, where $\Gamma = \{ \alpha \in \pi | g_{*}(\alpha) = \alpha, \forall g \in G \}.$
\item $\Im(\pi_{1}(K^G, \phi) \longrightarrow \pi_{1}(K, \phi)) = \Gamma$.
\end{enumerate}

\end{thm}

\begin{proof}
To prove the above theorem, we use the following version of Smith Theorem for Kan complexes, which may be proved using simplicial Bredon cohomology, following the method for $G$-CW complexes as described in \cite{May2}.

\noindent
\begin{thm}[\cite{Sm}] \label {smith}
 Let $G$ be a finite $p$-group, $p$ prime, acting on a finite one vertex Kan complex $K$ such that $H^{*} (K;\mathbb{Z}_{p})$ is
finite dimensional as a vector space over $\mathbb{Z}_{p}$. Further, let $H^{i}(K;\mathbb{Z}_{p})= 0$ for $i>0$, and $H^{0}(K;\mathbb{Z}_{p}) \cong \mathbb{Z}_{p}$.
Then, $H^{i}(K^G;\mathbb{Z}_{p})= 0$ for $i>0$, and $H^{0}(K^G;\mathbb{Z}_{p}) \cong \mathbb{Z}_{p}$.
\end{thm}

Since the action of $G$ on $K$ is simplicial, $(i)$ is clear. Let $\widetilde{K}$ denote the universal covering complex of $K$ and let $\tilde{\phi} = (\phi, e)$, where $e$ denotes the identity of $\pi$.
Then $\widetilde{K}$ is a contractible one vertex Kan complex and hence acyclic mod $\mathbb{Z}_{p}$.

\noindent
As shown in the earlier section, the action $(G,K,\phi)$ can be lifted to an action
$(G, \widetilde{K}, \tilde{\phi})$. Let $E$ := $\widetilde{K}^{G}$ denote the fixed point set of this action.
Then, by Smith Theorem \ref{smith}, $E$ is acyclic mod $\mathbb{Z}_{p}$.

\noindent
Now, we know $E/ \Gamma = K^G \subset K$. Furthermore, since the action of $G$ on $\widetilde{K}$ is given by $g \cdot (k, \alpha) = (g\cdot k, g_{*}(\alpha))$,
where $g \in G$, $ k \in K$ and $\alpha \in \pi$, we have $\Im(\pi_{1}(K^G, \phi) \longrightarrow \pi_{1}(K, \phi)) = \Gamma$.

Applying the Cartan-Leray spectral sequence \cite{Cartan} to $(E, \Gamma)$, we get a spectral sequence with $E_{2}$-term
\begin{displaymath}
 E_{2}^{p,q} = H^{p}(\Gamma; H^{q}(\mathbb{Z}_{p})),
\end{displaymath}
which converges to $H^{*}(K^G; \mathbb{Z}_{p})$. Therefore, by $\mathbb{Z}_{p}$-acyclicity, we have,
\begin{displaymath}
 H^{*}(K^G; \mathbb{Z}_{p}) \cong H^{*}(\Gamma; \mathbb{Z}_{p}).
\end{displaymath}

This proves the theorem.
\end{proof}

Notice that when $K$ is also minimal, we have the following corollary.
\begin{corollary} \label{cor54}
Let $G$ be a finite $p$-group action on a finite, one vertex, $K(\pi,1)-$complex $(K, \phi)$ with $H^*(K;\mathbb{Z}_p)$ 
finite dimensional as vector space over $\mathbb{Z}_{p}$. If the induced action of $G$ on $\pi_{1}(K, \phi)$ is trivial, 
then $G$ acts trivially on $K$.
\end{corollary}
\begin{proof} Under the hypothesis $\Gamma = \pi_{1}(K, \phi)$, hence $H^{*}(K^G; \mathbb{Z}_{p}) \cong H^{*}(K; \mathbb{Z}_{p})$.
Also, $K^G \subset K$. Therefore, since $K$ is a minimal Eilenberg-Maclane complex, we have $K^G = K$.
\end{proof}

As a corollary we get a version  of Borel's Theorem \cite[Theorem 3.2]{CR2} in the present context.
\begin{corollary} \label{borel}
Let $(K,\phi)$ be a finite, one vertex,  $K(\pi,1)-$complex, with finite dimensional $\mathbb{Z}_{p}$-cohomology, for all primes $p$.
Let $G$ be a finite group acting effectively on $K$. Then the associated abstract kernel, $\varphi \colon G \rightarrow \Aut(\pi)$, is a monomorphism.
\end{corollary}

\begin{proof}
 Suppose this is false. Then, for some prime $p$, $\ker{\varphi}$ contains a non-trivial Sylow $p$-subgroup, say, $L$. Then, applying Corollary \ref{cor54}, we see that the action of $L$ on $K$ is trivial, which contradicts the effectiveness of the action of $G$.
\end{proof}

\providecommand{\bysame}{\leavevmode\hbox to3em{\hrulefill}\thinspace}
\providecommand{\MR}{\relax\ifhmode\unskip\space\fi MR }
\providecommand{\MRhref}[2]{%
  \href{http://www.ams.org/mathscinet-getitem?mr=#1}{#2}
}
\providecommand{\href}[2]{#2}

\end{document}